\documentclass[12pt, a4paper]{article}

\usepackage{amsmath,amsfonts,amsthm,amssymb,mathrsfs}
\usepackage{enumerate}
\usepackage[margin=3cm]{geometry}
\usepackage[all]{xy}

\usepackage[bookmarks,colorlinks,linkcolor=black,citecolor=black]{hyperref}

\numberwithin{equation}{subsection}

\theoremstyle{plain}
\newtheorem{thm}{THEOREM}[section]
\newtheorem{definition}[thm]{DEFINITION}
\newtheorem{corol}[thm]{COROLLARY}
\newtheorem{lemma}[thm]{LEMMA}
\newtheorem{prop}[thm]{PROPOSITION}
\newtheorem{rem}[thm]{REMARK}

\DeclareMathOperator{\Ext}{Ext}
\DeclareMathOperator{\Hom}{Hom}

\newenvironment{enumerate-1}{%
\hspace{1cm}
\begin{itemize}%
\setlength{\itemsep}{0mm}%
\setlength{\parskip}{0.3ex plus 0.1ex minus 0.1ex}%
\setlength{\parsep}{0mm}%
\setlength{\partopsep}{0mm}%
}
{%
\end{itemize}%
}

\newcommand{\C}{$C^*$-algebra}
\newcommand{\G}{\mathcal{G}}
\newcommand{\F}{\mathcal{F}}
\newcommand{\K}{\mathcal{K}}
\newcommand{\R}{\mathbb{R}}

\title{\bf $K$-theory for the Leaf Spaces of the Orbit Foliations of the co-adjoint Action of some 5-dimensional Solvable Lie groups}
\author{Le Anh Vu\footnote{Faculty of Economic Mathematics, University of Economics and Law, Vietnam National University - Ho Chi Minh City, Viet Nam. E-mail: vula@uel.edu.vn} \,\, \and Nguyen Anh Tuan \,\, \and Duong Quang Hoa}
\date{}

\begin{document}

\maketitle

\begin{abstract}
	In this paper, combining Kirillov's method of orbits with Connes' method in Differential Geometry, we study the so-called MD(5,3C)-foliations, i.e. the orbit foliations of the co-adjoint action of MD(5,3C)-groups. First, we classify topologically MD(5,3C)-foliations based on the classification of all MD(5,3C)-algebras in \cite{Vu-Sh} and the picture of co-adjoint orbits ($K$-orbits) of all MD(5,3C)-groups in \cite{Vu-Th}. Finally, we study $K$-theory for leaf space of MD(5,3C)-foliations and describe analytically or characterize Connes' $C^*$-algebras of the considered foliations by $KK$-functors. 
\end{abstract}



\section*{Introduction}
	Let $G$ be an $n$-dimensional real solvable Lie group, then $G$ is called a MD$n$-group or a MD-group of dimension $n$ if and only if the orbits of $G$ in the co-adjoint representation (K-representation) are orbits of dimension zero or maximal dimension (i.e. dimension $k$, where $k$ is some even constant, $k \leq n$). When $G$ is an MD$n$-group such that its Lie algebra $\mathcal{G}$ has $m$-dimensional derived ideal ${\mathcal{G}}^{1}:= [\mathcal{G}, \mathcal{G}]$ \,($m < n$) then $G$ is called an MD($n,m$)-group and $\mathcal{G}$ is called an MD($n,m$)-algebra. In particular, if $\mathcal{G}^1$ is commutative then $G$ is called an MD($n,mC$)-group and $\mathcal{G}$ is called an MD($n,mC$)-algebra. The problem of classifying MD-algebras, describing the $K$-representation of MD-groups and characterizing the $C^*$-algebras of MD-groups is significant. Note that, if $n < 4$, all the n-dimensional Lie algebras are MD-algebras, and moreover, they can be listed easily. Therefore, we only take interests in MD$n$-algebras (and MD$n$-groups) for $n \geqslant 4$. If $G$ is a certain MD$n$-group, then the family $\mathcal{F}_{G}$ of its K-orbits of maximal dimension forms a foliation which is called is called MD$n$-foliation associated with $G$ (or with the Lie algebra Lie($G$) of $G$). When $G$ is an MD($n,m$)-group or MD($n,mC$)-group then $\mathcal{F}_{G}$ is also called MD($n,m$)-foliation or MD($n,mC$)-foliation, respectively.

In 1990, Vu classified all MD4-algebras, gave a topological classification of all MD4-foliations and described all Connes' $C^*$-algebras of them by using the $KK$-functors in 1992 (see \cite{Vu1}, \cite{Vu2}, \cite{Vu3}, \cite{Vu4}). Recently, Vu and his colleagues study MD5-algebras, MD5-groups and some related problems (see \cite{Vu-Ho07}, \cite{Vu-Sh}, \cite{Vu-Ho09}, \cite{Vu-Ho10}, \cite{Vu-Ho11}, \cite{Vu-Hi-Ng}). Up to the present moment, there is no complete classification for MD$n$-algebras with $n > 5$. In \cite{Vu-Ho-Tu}, we considered all MD(5,4)-algebras and analytically described or characterized Connes' $C^*$-algebras of MD(5,4)-foliations by using $KK$-functors.

This paper is a continuation of authors' works \cite{Vu-Ho09}, \cite{Vu-Ho10}, \cite{Vu-Ho11}, \cite{Vu-Ho-Tu}. Basing on the results of \cite{Vu-Sh} and \cite{Vu-Th}, we consider MD(5,3C)-foliations. The paper is organized as follows: Section 1 deals with the classification of all MD(5,3C)-algebras (\cite{Vu-Sh}). Section 2 studies MD(5,3C)-foliations associated to classified MD(5,3C)-algebras. Section 3 is devoted to study $K$-theory for the leaf space of considered MD(5,3C)-foliations and characterize the Connes' $C^*$-algebras associated to these foliations.

In another paper, we will consider the analogue problem for the subclass of MD(5,3)-foliations and the entire class of MD5-foliations.

\section{The class of MD5-algebras having 3-dimensional commutative dereived ideals}

	In this section, we recall the classification of all MD(5,3C)-algebras.

	Let $G$ be a Lie group and $\G = {\rm Lie}(G)$ be the Lie algebra of $G$. We use $\G^*$ to denote the dual space of $\G$. For every $g \in G$, we denote the internal automorphism associated with $g$ by $A_g$, and whence, $A_g: G \rightarrow G $ can be defined as follows:
		\[
			A_g (x) := g.x.g^{-1},\, \forall x \in G.
		\]
This above automorphism induces the following mapping:
		\[
		{A_{(g)}}_{*}:\mathcal{G}\longrightarrow \mathcal{G}\qquad \qquad \qquad \qquad $$
$$\qquad\qquad \qquad \qquad\qquad X\longmapsto {A_{(g)}}_{*}(X):\, =\;\frac{d}{dt}
[g.exp(tX)g^{-1}]\mid_{t=0}
		\]
which is called \emph{the tangent mapping} of $A_{(g)}$.

We now formulate the following definitions.

\begin{definition}The action
$$Ad :G\longrightarrow Aut(\mathcal{G})$$
$$\qquad\qquad\qquad g\longmapsto Ad(g):\, =\;  {A_{(g)}}_{*}$$
is called the adjoint representation of G in $\mathcal{G}$.\\
\end{definition}

\begin{definition}The action
$$K:G\longrightarrow Aut(\mathcal{G}^{*})$$
$$g\longmapsto K_{(g)}$$
such that
$$\langle K_{(g)}F,X\rangle :\,=\langle F, Ad(g^{-1})X\rangle ;
\quad(F\in {\mathcal{G}}^{*},\, X \in \mathcal{G})$$ is called the
co-adjoint representation or K-representation of G in $\mathcal{{G}^{*}}$.\\
\end{definition}

\begin{definition} Each orbit of the co-adjoint representation of $G$ is called a $K$-orbit of $G$.
\end{definition}

Thus, for every $F \in \mathcal{G}^{*}$, the K-orbit containing $F$ defined above can be written by
$$\Omega_F:= \{K_{(g)}F / g \in G \}.$$

\begin{definition}[{{\bf \cite[Chapter 4, Definition 1.1]{Di99}}}]
	An n-dimensional MD-group (for brevity, MDn-group) is an n-dimensional real solvable Lie group such that its K-orbits are orbits of dimension zero or maximal dimension. The Lie algebra of an MDn-group is called an n-dimensional MD-algebra (for short, MDn-algebra).
\end{definition}

	Recently, the class of all MD5-algebras has been completely classified (see \cite{Vu-Sh}, \cite{Vu-Hi-Ng}), up to isomorphism. In this paper, we just consider a subclass of MD(5,3C)-algebras, i.e. MD5-algebras having 3-dimensional commutative derived ideals.

\begin{prop}[{\cite{Vu-Sh}}]\label{MD(5,3C)-algebras}
	Let $\G$ be an indecomposable MD5-algebra whose $\G^1:=[\G, \G] \cong \R^3$. Then we can choose a suitable basis $\left\lbrace X_1,X_2,X_3,X_4,X_5 \right\rbrace$ of $\G$ such that $\G^1 = \R.X_3 \oplus \R.X_4 \oplus \R.X_5 \equiv \R^3$; $[X_{1},X_{2}]= X_{3}$,$ad_{X_1} = 0 $, $ad_{X_2} \in {\rm End}\left(\G^1\right) \equiv {\rm Mat}_3(\R)$ and $\G$ is isomorphic to one and only one of the following Lie algebra.

    \begin{itemize}
     \item[\bf 1.]${\mathcal{G}}_{5,3,1({\lambda}_{1}, {\lambda}_{2})}:\,
 ad_{{X}_2} = \begin{pmatrix} {{\lambda}_1}&0&0\\
0&{{\lambda}_2}&0\\0&0&1 \end{pmatrix}; \quad {\lambda}_1,
{\lambda}_2 \in \R\setminus \lbrace 0, 1\rbrace, \, {\lambda}_1
\neq {\lambda}_2.$

        \item[\bf 2.]${\mathcal{G}}_{5,3,2(\lambda)}:\, 
ad_{{X}_2} = \begin{pmatrix} 1&0&0\\
0&1&0\\0&0&{\lambda} \end{pmatrix}; \quad {\lambda} \in
\R\setminus \lbrace 0, 1 \rbrace .$ 

        \item[\bf 3.]${\mathcal{G}}_{5,3,3(\lambda)}:\,
ad_{{X}_2} = \begin{pmatrix} {\lambda}&0&0\\ 0&1&0\\0&0&1 \\ \end{pmatrix};
\quad {\lambda} \in \R\setminus \lbrace 1 \rbrace . $

        \item[\bf 4.]${\mathcal{G}}_{5,3,4}:\,
ad_{{X}_2} = \begin{pmatrix} 1&0&0\\
0&1&0\\0&0&1 \end{pmatrix}.$

        \item[\bf 5.]${\mathcal{G}}_{5,3,5(\lambda)}:\,
 ad_{{X}_2} = \begin{pmatrix} {\lambda}&0&0\\
 0&1&1\\0&0&1 \end{pmatrix}; \quad{\lambda} \in \R\setminus \lbrace 1 \rbrace .$

        \item[\bf 6.]${\mathcal{G}}_{5,3,6(\lambda)}:\,
 ad_{{X}_2} = \begin{pmatrix} 1&1&0\\
 0&1&0\\0&0&{\lambda} \end{pmatrix}; \quad
 {\lambda} \in \R\setminus \lbrace 0, 1 \rbrace .$

        \item[\bf 7.]${\mathcal{G}}_{5,3,7}:\,
 ad_{{X}_2} = \begin{pmatrix} 1&1&0\\
 0&1&1\\0&0&1 \end{pmatrix}.$

        \item[\bf 8.]${\mathcal{G}}_{5,3,8(\lambda, \varphi)}:\,
 ad_{{X}_2} = \begin{pmatrix} cos{\varphi}&-sin{\varphi}&0\\
 sin{\varphi}&cos{\varphi}&0\\0&0&\lambda \end{pmatrix}; \quad
 \lambda \in \R\setminus \lbrace 0\rbrace, \, \varphi \in (0, \pi) .$
        \end{itemize}
\end{prop}

\begin{rem}
		Because each real Lie algebra $\G$ defines only one connected and simply connected Lie group $G$ such that ${\rm Lie}(G)= \G$. Therefore, we obtain a collection of 8 families of connected and simply connected MD(5,3C)-groups corresponding to the indecomposable MD(5,3C)-algebras given in Proposition \ref{MD(5,3C)-algebras}. For convenience, each MD(5,3C)-group from this collection is also denoted by the same indices as corresponding MD(5,3C)-algebra. For example, $G_{5,3,1(\lambda_1,\lambda_2)}$ is the indecomposable connected and simply connected MD(5,3C)-group corresponding to $\G_{5,3,1(\lambda_1,\lambda_2)}$.
	\end{rem}

\section{The class of MD(5,3C)-foliations associated to the considered MD(5,3C)-groups}

	In this section, we introduce foliations formed by the generic $K$-orbits of the connected and simply connected MD(5,3C)-groups. These foliations are called, for brevity, MD(5,3C)-foliations. First, we recall the geometric picture of the co-adjoint orbits of all MD(5,3C)-groups in \cite{Vu-Th}.
	
	Let $G$ be one of MD(5,3C)-groups, i.e. $G$ is one of connected and simply connected MD(,3C)5-groups corresponding to the MD(5,3C)-algebras listed in Proposition \ref{MD(5,3C)-algebras}. Denote by ${\mathcal{G}}^{*}$ the dual space of the Lie algebra $\mathcal{G} = {\rm Lie} (G)$ of $G$. Clearly,  ${\mathcal{G}}^{*}$ can be identified with ${\R}^{5}$ by fixing in it the dual basis $\left\lbrace X_{1}^{*}, X_{2}^{*}, X_{3}^{*}, X_{4}^{*}, X_{5}^{*} \right\rbrace$ of the basis $(\nobreak X_{1}, X_{2}, X_{3}, X_{4}, X_{5} \nobreak)$ of $\mathcal{G}$. Let
$F = {\alpha}X_{1}^{*} + {\beta}X_{2}^{*} + {\gamma}X_{3}^{*} + {\delta}X_{4}^{*} + {\sigma}X_{5}^{*} \equiv ({\alpha}, {\beta}, {\gamma}, {\delta}, {\sigma})$ be an arbitrary element of ${\mathcal{G}}^{*} \equiv {\R}^{5}$. The notation ${\Omega}_{F}$ will be used to denote the $K$-orbit of $G$ which contains $F$. The geometric picture of the $K$-orbits of $G$ is given by the following proposition.

\begin{prop}[{\cite{Vu-Th}}]
	For each MD(5,3C)-group $G$, the $K$-orbit $\Omega_F$ of $G$ is described as follows.
\begin{description}
        \item [1.] Let $G$ is one of $G_{5,3,1({\lambda}_{1}, {\lambda}_{2})}$, $G_{5,3,2(\lambda)}$, $G_{5,3,3(\lambda)}$, $G_{5,3,4}$, $G_{5,3,5(\lambda)}$, $G_{5,3,6(\lambda)}$, $G_{5,3,7}$.
        \begin{description}
            \item[1.1.] If $\beta=\gamma=\delta=\sigma=0$ then $\Omega_F=\{F\}$ (the 0-dimensional orbit).
            \item[1.2.] If $\beta^2+\gamma^2+\delta^2+\sigma^2\neq0$ then $\Omega_F$ is the orbit of dimension 2 and it is one of the following:
        \end{description}
\begin{itemize}
    \item $\left\lbrace \left(\alpha+\frac{1-e^{\lambda_1a}}{\lambda_1}\gamma; y; e^{\lambda_1a}\gamma; e^{\lambda_2a}\delta; e^a\sigma\right): y,a \in \R \right\rbrace$ when $G=G_{5,3,1({\lambda}_{1},{\lambda}_{2})}$.
    \item $\left\lbrace \left(\alpha+\left(1-e^a\right)\gamma; y; e^a\gamma; e^a\delta; e^{\lambda a}\sigma\right): y,a \in \R \right\rbrace$ when $G=G_{5,3,2(\lambda)}$.
    \item $\left\lbrace \left( \alpha + \frac{1 - e^{\lambda a}}{\lambda}\gamma; y; e^{\lambda a}\gamma; e^a\delta; e^a\sigma\right): y, a \in \R \right\rbrace$ when $G=G_{5,3,3(\lambda)}$.
    \item $\left\lbrace \left(\alpha+\left(1-e^a\right)\gamma; y; e^a\gamma; e^a\delta; e^a\sigma\right): y, a \in \R \right\rbrace$ when $G=G_{5,3,4}$.
    \item $\left\lbrace \left(\alpha+\frac{1-e^{\lambda a}}{\lambda}\gamma; y; e^{\lambda a}\gamma; e^a\delta; ae^a\delta+e^a\sigma\right): y, a \in \R \right\rbrace$ when $G=G_{5,3,5(\lambda)}$.
    \item $\left\lbrace \left(\alpha+\left(1-e^a\right)\gamma; y; e^a\gamma; ae^a\gamma+e^a\delta; e^{\lambda a}\sigma\right): y, a \in \R \right\rbrace$ when $G=G_{5,3,6(\lambda)}$.
    \item $\left\lbrace \left(\alpha+\left(1-e^a\right)\gamma; y; e^a\gamma; ae^a\gamma+e^a\delta; \frac{a^2e^a}{2}\gamma+ae^a\delta+e^a \sigma\right): y, a \in \R \right\rbrace$ when $G=G_{5,3,7}$.
\end{itemize}
   \item [3.] Let $G$ is $G_{5,3,8(\lambda, \varphi)}$. Let us identify $\G_{5,3,8(\lambda,\varphi)}^*$ with
$\R^2 \times \mathbb{C} \times \R$ and $F$ with $(\alpha,\beta,\gamma+i\delta,\sigma)$. Then we have:
\begin{description}
    \item[3.1.] If $\gamma+i\delta=\sigma=0$ then $\Omega_F=\{F\}$ (the 0-dimensional orbit).
    \item[3.2.] If $\beta^2+\gamma^2 \neq 0 \neq \sigma$ then $\Omega_F$ is the orbit of dimension 2 as follow: $$\Omega_F= \left\lbrace \left(x; y; \left(\gamma+i\delta\right)e^{ae^{-i\varphi}}; e^{\lambda a}\sigma\right): y, a \in \R \right\rbrace.$$
\end{description}
\end{description} 
\end{prop}

	In the introduction we have emphasized that, for every connected and simply connected MD-group, the family of maximal dimensional $K$-orbits forms a measurable foliation in terms of Connes (\cite{Co82}). Namely, we have the following proposition.

\begin{prop}[{\cite{Vu-Th}}]\label{MD(5,3C)-foliations}
	Let $G$ be one of the connected and simply connected MD(5,3C)-groups, $\mathcal{F}_{G}$ be the family of all its $K$-orbits of dimension two and $V_G: = \bigcup \{ \Omega / \Omega \in \mathcal{F}_{G}\}$. Then $\left(V_{G},\F_G \right)$ is a measurable foliation in the term of Connes. We call it \emph{MD(5,3C)-foliation associated to MD(5,3C)-group $G$}.
\end{prop}

\begin{rem}
	According to Proposition \ref{MD(5,3C)-algebras} and Proposition \ref{MD(5,3C)-foliations}, we obtain 8 families of MD(5,3C)-foliations associated to 8 families of MD(5,3C)-groups. Note that, for each MD(5,3C)-group $G$, the foliated manifold $V_G$ is an open submanifold of the dual space $\G^* \cong \R^5$ of the Lie algebra $\G$ corresponding to $G$. Furthermore, for all MD(5,3C)-groups of the forms $G_{5,3,\ldots}$, the manifolds $V_G$ are diffeomorphic to each other. So, for simplicity of notation, we shall write $(V,\F_{3,\ldots})$ instead of $(V_{G_{5,3,\ldots}},\F_{G_{5,3,\ldots}})$. 
\end{rem}

	Now we give the topological classification of 8 families of MD(5,3C)-foliations in the following theorem.

\begin{thm}[{{\bf The classification of MD(5,3C)-foliations}}]\label{MD(5,3C)-foliations' classification}
\begin{enumerate-1}
    \item[\bf 1.] There exist exactly 2 topological types of 8 families of considered MD(5,3C)-foliations as follows:
    \begin{enumerate}
        \item[\bf 1.1.] $\left\lbrace \left(V,\mathcal{F}_{1 \left(\lambda_1, \lambda_2 \right)}\right), \left(V,\mathcal{F}_{2 \left(\lambda \right)}\right), \left(V,\mathcal{F}_{3 \left(\lambda \right)}\right), \left(V,\mathcal{F}_{4}\right), \left(V,\mathcal{F}_{5 \left(\lambda \right)}\right), \left(V,\mathcal{F}_{6 \left(\lambda \right)}\right), \left(V,\mathcal{F}_{7}\right) \right\rbrace$.
        \item[\bf 1.2.] $\left\lbrace \left(V,\mathcal{F}_{8 \left(\lambda, \varphi \right)}\right) \right\rbrace$.
    \end{enumerate}
    We denote these types by $\mathscr{F}_1$ and $\mathscr{F}_2$, respectively.
    \item[\bf 2.] Furthermore, we have:
    \begin{enumerate}
        \item[\bf 2.1.] The MD(5,3C)-foliations of the type $\mathscr{F}_1$ are trivial fibration with connected fibres on $\R \times S^2$.
        \item[\bf 2.2.] The MD(5,3C)-foliations of the type $\mathscr{F}_2$ are non-trivial but they can be given by suitable actions of $\R^2$ on the foliated manifolds $V \cong \R^2 \times \left(\mathbb{C} \times \R\right)^*$.
    \end{enumerate}
\end{enumerate-1} 
\end{thm}

\begin{proof}
		Let us recall that two foliations $(V, \F)$ and $(V', \F')$ are said to be \emph{topologically equivalent} or have \emph{same topological type} if there exist a homeomorphism $h: V \to V'$ which sends each leaves of $\F$ onto those of $\F'$. The map $h$ is called a \emph{topological equivalence} of considered foliations.
		 \begin{enumerate}
		 	\item[\bf 1.] We consider maps $h_{1\left( \lambda_1, \lambda_2 \right)}, h_{2 \left( \lambda \right)}, h_{3 \left( \lambda \right)}, h_{5 \left( \lambda \right)}, h_{6 \left( \lambda \right)}, h_7$ from $V \cong \R^2 \times \left(\R^3\right)^*$ to $V$ which are defined as follows:
				\[
						\begin{array}{l}
							h_{1\left( \lambda_1, \lambda_2 \right)} \left( x; y; z; t; s \right) =
	 									\left( \lambda_1 x + z - {\rm sgn}(z). {|z|}^{\frac{1}{\lambda_1}}; y; {\rm sgn}(z). {|z|}^{\frac{1}{\lambda_1}}; {\rm sgn}(t). {|t|}^{\frac{1}{\lambda_2}}; s \right) \\
	 						h_{2 \left( \lambda \right)} \left( x; y; z; t; s \right)= \left( x; y; z; t; {\rm sgn}(s).{|s|}^{\frac{1}{\lambda}} \right) \\
	 						
	 						h_{3\left( \lambda \right)} \left( x; y; z; t; s \right) := \begin{array}{l l}
	 								\left( \lambda x + z - {\rm sgn}(z). {|z|}^{\frac{1}{\lambda}}; y; {\rm sgn}(z). {|z|}^{\frac{1}{\lambda}}; t; s \right)
	 							& \left( \lambda \neq 0 \right)
	 						\end{array} \\
	 						h_{5 \left( \lambda \right)} \left( x; y; z; t; s \right) :=
	 							\left\lbrace
	 								\begin{array}{l l}
										\left(x; y; z; t; s - t \ln|t| \right) & \left( t \ne 0 \right) \\
										\left( x; y; z; 0; s \right) & \left( t = 0 \right)
									\end{array}
								\right. \\
								h_{6 \left( \lambda \right)} \left( x; y; z; t; s \right) :=
	 							\left\lbrace
	 								\begin{array}{l l}
	 									\left( x; y; z; t - z \ln|z|; {\rm sgn}(s). {|s|}^{\frac{1}{\lambda}} \right)
	 										& \left( z \neq 0 \right) \\
	 									\left( x; y; 0; t; {\rm sgn}(s). {|s|}^{\frac{1}{\lambda}} \right)
	 										& \left( z=0 \right) \\
	 								\end{array}
	 							\right. \\
	 							h_7 \left( x; y; z; t; s \right) := \left( x; y; z, \tilde{t}, \tilde{s} \right),
						\end{array}
					\]
				where:
					\[
	 					\begin{array}{l}
	 						\tilde{t} := \left\lbrace
	 									\begin{array}{l l}
	 										0 & \left( z = t = 0 \vee  z \neq 0, t = z \ln|z| \right) \\
	 										t & \left( z = 0; t \neq 0 \right) \\
	 										t - z \ln|z| & \left( z \neq 0, t \neq z \ln|z| \right)
	 									\end{array} \right. \\
	 						\tilde{s} := \left\lbrace
	 									\begin{array}{l l}
	 										s & \left( z = t = 0 \right) \\
	 										s - t \ln|t| & \left( z = 0; t \neq 0 \right) \\
	 										s - \frac{1}{2} t \ln|z| & \left( z \neq 0, t = z \ln|z| \right) \\
	 										s - \frac{1}{2} t \ln|z| - \frac{1}{2} \left( t - z \ln|z| \right). \ln|t - z \ln|z|| & \left( z \neq 0, t \neq z \ln|z| \right)
	 									\end{array} \right.
	 						\end{array}	
	 					\]
					By some direct caculations, we have:
					\begin{itemize}
						\item $h_{1\left( \lambda_1, \lambda_2 \right)}$, $h_{2 \left( \lambda \right)}$, $h_{5 \left( \lambda \right)}$, $h_{6 \left( \lambda \right)}$, $h_7$ are homeomorphisms send leaves of $\left(V,\mathcal{F}_{1 \left(\lambda_1, \lambda_2 \right)}\right)$, $\left(V,\mathcal{F}_{2 \left(\lambda \right)}\right)$, $\left(V, \mathcal{F}_{5 \left(\lambda \right)}\right)$, $\left(V, \mathcal{F}_{6 \left(\lambda \right)}\right)$, $\left(V,\mathcal{F}_7\right)$ onto leaves of $\left(V, \mathcal{F}_{4}\right)$.
						\item If $\lambda \neq 0$ then $h_{3\left( \lambda \right)}$ is an homeomorphism sends leaves of $\left(V, \mathcal{F}_{3 \left(\lambda \right)}\right)$ onto leaves of $\left(V, \mathcal{F}_{4}\right)$. If $\lambda = 0$, by the geometric picture of $K$-orbits, we see that all the maximal dimension $K$-orbits of $\left(V, \mathcal{F}_{3 \left(0 \right)}\right)$ and $\left(V, \mathcal{F}_{4}\right)$ are half-planes; so, they are obviously homeomorphis.
					\end{itemize}
			So, the foliations $\left(V,\mathcal{F}_{1 \left(\lambda_1, \lambda_2 \right)}\right)$, $\left(V,\mathcal{F}_{2 \left(\lambda \right)}\right)$, $\left(V,\mathcal{F}_{3 \left(\lambda \right)}\right)$, $\left(V,\mathcal{F}_{4}\right)$, $\left(V,\mathcal{F}_{5 \left(\lambda \right)}\right)$, $\left(V,\mathcal{F}_{6 \left(\lambda \right)}\right)$, $\left(V,\mathcal{F}_{7}\right)$ are topologically equivelant. This type is denoted by $\mathscr{F}_1$.
			
		Now, by direct calculations, we see that $h_{8 \left( \lambda, \varphi \right)}: V \cong \R^2 \times (\mathbb{C} \times \R)^*  \to V$ defined by:
			\[
	 		h_{8 \left( \lambda, \varphi \right)} \left( x; y; re^{i\theta}; s \right) = \left( \tilde{x}; y; \tilde{r}e^{i\tilde{\theta}}; {\rm sgn} (s).{|s|}^{\frac{1}{\lambda}} \right)
	 		\]
	where:
	 	\[
			\begin{array}{l}
				\tilde{x} :=
							\begin{cases}
								x + r \cos \left(\theta + \varphi \right) + {\rm Im} \left[ e^{\left(\ln r + i\theta \right) \left( -ie^{i\varphi} \right)} \right] & \left(r \neq 0\right) \\
								x & \left(r = 0\right)
							\end{cases} \\
				\tilde{r}e^{i\tilde{\theta}} :=
														\begin{cases}
															e^{\left(\ln r + i\theta \right) \left( -ie^{i\varphi} \right)} & \left(r \neq 0\right) \\
															0 & \left(r = 0\right)
														\end{cases}
			\end{array}
		\]
			is a homeomorphism which sends leaves of $\left(V,\mathcal{F}_{8 \left( \lambda, \varphi \right)} \right)$ onto leaves of $\left(V,\mathcal{F}_{8 \left( 1, \frac{\pi}{2} \right)}\right)$. So, the foliations $\left(V,\mathcal{F}_{8 \left( \lambda, \varphi \right)} \right)$ are topologically equivelant. This types is denoted by $\mathscr{F}_2$.
		  \item[\bf 2.] By using stratergic coordinate, we have $\left( \R^3 \right)^* \cong \R_+ \times S^2$. The submersion $p: V \cong \R^2 \times \left( \R^3 \right)^* \cong \R^2 \times \R_+ \times S^2 \rightarrow \R \times S^2$ defined by:
			\[
				p(x;y; r; \phi; \theta) := (x + r\cos \phi \sin \theta; \phi; \theta)
			\]
is a fibre bundle which each fibre is (isomorphic to) $\R \times \R_+$. Each fibre is connective (half-plane) and that is each leaf of the foliation $\left(V,\mathcal{F}_{4}\right)$. So, the MD(5,3C)-foliations belong to $\mathscr{F}_1$-type come from fibre bundle on $\R \times S^2$.

Now, we consider continuos map $\rho: \R^2 \times V \to V$ defined by:
				\[
					\rho \left( \left( r; a \right); \left( x; y; z+it; s \right) \right) = \left( x - \left( \sin a \right) z - \left( 1 - \cos a \right) t; y + r; \left( z + it \right)e^{-ia}; e^a s \right)
				\]
			It is easy to check that the Lie group $\R^2$ acts continuously on the foliated manifold $V$ by $\rho$. Morever, under $\rho$-action, orbit through $\left( \alpha; \beta; \gamma + i \delta; \sigma \right) \in V$ is:
				\begin{equation}\label{rho-action}
					\rho_{\left( \alpha; \beta; \gamma + i \delta; \sigma \right)} = \left\lbrace \left( \alpha - \left( \sin a \right) \gamma - \left( 1 - \cos a \right) \delta; \beta + r; \left( \gamma + i \delta \right)e^{-ia}; e^a \sigma \right): y, a \in \R \right\rbrace
				\end{equation}
and that is leave of the foliation $\left( V, \mathcal{F}_{8 \left(1, \frac{\pi}{2}\right)}\right)$. So, the MD(5,3C)-foliations belong to $\mathscr{F}_2$-type are given by $\rho$-action. The proof is complete.
	\end{enumerate}
\end{proof}

\section{$K$-theory for the MD(5,3C)-foliations}

	In this section, we study $K$-theory for the leaf space of MD(5,3C)-foliations. First, we introduce something about $K$-theory and the method of $KK$-functors to characterize $C^*$-algebras.

\subsection{$K$-theory and characterization of $C^*$-algebras by $KK$-functor}

	Topological $K$-theory is a generalized cohomology theory with compact supports. It is well adapted to the algebraic framework where the locally compact space $X$ is replaced by the {\C} $C_0(X)$ of all the complex-valued continuous functions on $X$ vanishing at infinity. The group $K_*(A)$ makes sense for any {\C} $A$ and one has $K^*(X) \cong K_*\left(C_0(X)\right)$ (see \cite{Ro-La-La} or \cite{Ta}). Note that the $K$-theory of $C^*$-algebras is invariant under stable isomorphism, i.e., $K_*(A\otimes \K) \cong K_*(A)$ for every {\C} $A$. 

	In algebraic $K$-theory, each short exact sequence (or \emph{an extension}) of $C^*$-algebras:
\begin{equation}\label{SES}
   \xymatrix{0 \ar[r] & J \ar[r]^i & A \ar[r]^{\mu} & B \ar[r] & 0}
\end{equation} 
gives rise to a six-term circulative exact sequence (\cite{Ro-La-La}):
\begin{equation}\label{6-ES}
        \xymatrix{K_0(J)\ar[r]^{i_*} & K_0(A)\ar[r]^{\mu_*} & K_0(B)\ar[d]^{\delta_0}\\
        K_1(B)\ar[u]^{\delta_1} & K_1(A) \ar[l]_{\mu_*} & K_1(J)\ar[l]_{i_*}} 
\end{equation}
where $\delta_0$ and $\delta_1$ are called \emph{the connecting maps}. When $J$ and $B$ in \eqref{SES} are given $C^*$-algebras, but $A$ is an unknown one, calculation of the connecting maps of \eqref{6-ES} gives the $K$-groups $K_0(A)$ and $K_1(A)$. 

In fact, we obtain more by this calculation, and this may be seen as another motivation for studying $K_*(A)$. The extension \eqref{SES} defines $A$ as an element in $KK$-group $\Ext(B, J)$ of Kasparov \cite{Kas}. Namely, by the universal coefficient theorem (see \cite{Ro}) we have the following exact sequence ($j \in \mathbb{Z}/2\mathbb{Z})$:
	\begin{equation}\label{UCT}
   	\xymatrix{
   	0 \to \oplus \Ext^1_{\mathbb{Z}} \left(K_j(B), K_j(J)\right) \ar[r] & \Ext (B,J) \ar[r]^{\hspace{-1.5cm} \gamma} & \oplus \Ext^1_{\mathbb{Z}} \left(K_j(B), K_{j+1}(J)\right) \to 0}
	\end{equation} 
In \eqref{UCT}, the map $\gamma$ associates to the class in $\Ext(B, J)$ defined by the short exact sequence \eqref{SES} the pair $(\delta_0, \delta_1)$ of the diagram \eqref{6-ES}. $\Ext^1_\mathbb{Z}$ is defined as in homological algebra (see \cite{Ma}). In this paper, $K_j(B)$ and $K_j(J)$ are always free abelian groups, hence $\Ext^1_\mathbb{Z} \left(K_j(B), K_j(J) \right) = 0$ for $j = 0,1$. Thus, the pair $(\delta_0, \delta_1)$ determines {\C} $A$ as an element of $\Ext(B, J)$. In general, this is only sufficient to determine the so-called \emph{``stable type''} of the extension \eqref{SES}. Morever, when the extension is absorbing, the pair $(\delta_0, \delta_1)$ determines \eqref{SES} up to unitary equivalence (see \cite[Section 7]{To}). So, the pair $(\delta_0, \delta_1)$ is called the \emph{index invariant} of {\C} $A$. 

Generally, if an unknown {\C} $A$ is embedded in two repeated extensions of the following form:
\begin{equation}\label{RSES-1}
   \xymatrix{0 \ar[r] & J_1 \ar[r]^{i_1} & A \ar[r]^{{\mu}_1} & B_1 \ar[r] & 0}
\end{equation}
\begin{equation}\label{RSES-2}
   \xymatrix{0 \ar[r] & J_2 \ar[r]^{i_2} & B_1 \ar[r]^{{\mu}_2} & B_2 \ar[r] & 0}
\end{equation}
where $J_1, J_2$ and $B_2$ are given $C^*$-algebras, then two pairs of the connecting maps of the six-term circulative exact sequences associated to the extensions \eqref{RSES-1} and \eqref{RSES-2} are called the \emph{index invariant system} of $A$.
    
In the decades 1970s-1980s, works of Diep \cite{Di75}, Rosenberg \cite{Ro}, Kasparov \cite{Kas}, Son and Viet \cite{So-Vi},\ldots have shown that $K$-functors are well adapted to characterize a large class of group $C^*$-algebras.

\subsection{The Connes' $C^*$-algebras}

Recall that, in 1982, studying foliated manifolds, Connes \cite{Co82} introduced the notion of {\C} $C^*(V, \F)$ associated to a measured foliation $(V, \F)$. In general, the leaf space $V/\F$ of a foliation $(V, \F)$ with the quotient topology is a fairly untractable topological space. The Connes' {\C} $C^*(V, \F)$ represents the leaf space $V/\F$ of $(V, \F)$ in the following sense: when the foliation $(V, \F)$ comes from a fibration (with connected fibers) $p: V \to B$ on some locally compact basis $B$, in particular the leaf space of $(V, \F)$ in this case is the basis $B$, then $C^{*}(V, \F)$ is isomorphic to ${C_0}(B)\otimes \K$, where $\K = \K(\mathcal{H})$ denotes the {\C} of compact operators on an infinite dimensional separable Hilbert space $\mathcal{H}$. For such foliations $(V, \F)$, $K_{*}{\left( C^*(V, \F)\right)}$ coincides with the $K$-theory of the leaf space $B = V/\F$. Therefore, for an arbitrary foliation $(V, \F)$, Connes \cite{Co82} used $K_* \left( C^*(V, \F) \right)$ to define the $K$-theory for the leaf space of $(V, \F)$.

For a given foliation $(V, \F)$, the Connes' {\C} $C^*(V, \F)$ is generated by certain functions on the graph (or holonomy groupoid) $G$ (\cite[Section 2]{To}) of the foliations $(V, \F)$. The construction of $C^*(V, \F)$ is fairly complete (see \cite[Sections 5-6]{Co82}), so we do not recall this construction here. Some properties of the Connes' {\C} are introduced below.

\begin{prop}[{\cite[Proposition 2.1.4]{To}}]\label{Connes-1}
Topologically equivalent foliations yield isomorphic $C^*$-algebras.
\end{prop}

\begin{prop}[{\cite[Proposition 2.1.5]{To}}]\label{Connes-2}
Assume that the foliation $(V, \F)$ comes from an action of a Lie $H$ in such a way that the graph G is given as $V \times H$. Then $C^*(V, \F)$ is isomorphic to the reduced crossed product $C_0(V)\rtimes H$ of $C_0(V)$ by $H$.
\end{prop}

\begin{prop}[{\cite[Proposition 2.1.6]{To}}]\label{Connes-3}
Assume that the foliation $(V, \F)$ is given by a fibration (with connected fibers) $p: V \to B$. Then the graph 
$G = \lbrace (x, y) \in V \times V | p(x) = p(y) \rbrace $ which is a submanifold of $V \times V$ and $C^*(V, \F)$ is isomorphic to 
$C_0(B)\otimes \K$.
\end{prop}

\begin{prop}[{\cite[Proposition 2.1.7]{To}}]\label{Connes-4}
Let $(V, \F)$ be a foliation. If $V' \subset V$ is an open subset of foliated manifold V and ${\F}^{'}$ is the restriction of $\F$ to $V^{'}$. Then the graph $G'$ of $(V', \F')$ is an open subset in the graph $G$ of $(V, \F)$ and $C^*(V', \F')$ can be canonically embedded into $C^*(V, \F)$.
\end{prop}

\begin{prop}[{\cite[Subsection 2.2]{To}}]\label{Connes-5}
Let $V' \subset V$ be a saturated open subset in the foliated manifold $V$ of a foliation $(V, \F)$. Denote by $\F', \F"$ the restrictions of $\F$ on $V'$ and $V \setminus V'$, respectively. Then $C^*(V, \F')$ is an ideal in $C^*(V, \F)$. Moreover, if the foliation $(V, \F)$ is given by a suitable action of an amenable Lie group $H$ such that the subgraph $G \setminus G'$ is given as $(V \setminus V') \times H$ then $C^*(V, \F)$ can be canonically embedded into the following exact sequence:
\begin{equation}
\xymatrix{0 \ar[r] & C^*\left( V, \F' \right) \ar[r] & C^*\left( V, \F \right) \ar[r] & C^*\left( V \setminus V', \F" \right) \ar[r] & 0}. 
\end{equation}
\end{prop}

By Proposition \ref{Connes-1}, for brevity, we denote by $C^* \left( \mathscr{F}_1 \right)$ and $C^* \left( \mathscr{F}_2 \right)$ the Connes' $C^*$-algebras assoiated to MD(5,3C)-foliations belong to $\mathscr{F}_1$-type and $\mathscr{F}_2$-type, respectively. Morever, as a direct consequence of Theorem \ref{MD(5,3C)-foliations' classification} and Propositions \ref{Connes-1} $\div$ \ref{Connes-3}, we have the following assertion.     

\begin{corol}[{{\bf Analytical description of $C^* \left( \mathscr{F}_1 \right)$ and $C^* \left( \mathscr{F}_2 \right)$}}]\label{Analytical description}
\begin{enumerate-1}
   \item[\bf 1.] The Connes' $C^*$-algebras of all MD(5,3C)-foliations belong to the $\mathscr{F}_1$-type are isomorphic to $C_0 \left(\R \times S^2\right) \otimes \K$. 
   \item[\bf 2.] The Connes' $C^*$-algebras of all MD(5,3C)-foliations belong to the $\mathscr{F}_2$-type are isomorphic to the reduced crossed product $C_0(V)\rtimes_\rho \R^2$ where $V \cong \R^2 \times \left(\R^3\right)^* \cong \R^2 \times \left( \mathbb{C} \times \R \right)^*$.
\end{enumerate-1}
\end{corol}

\subsection{$C^*(\mathscr{F}_2)$ as extension of $C^*$-algebras}

	Many foliations are defined by decomposing suitably the given manifold, constructing foliations on the smaller pieces and then patching together to obtain a foliation of the total manifold. If the pieces when included in the total manifold are saturated with respect to the foliation, this may be viewed as patching together bits of leaf space. Translating into the language of $C^*$-algebras, this construction of the foliation $(V, \F)$ gives us an extension of the form \eqref{SES}. Then, to study the $K$-theory of the leaf space of this foliation $(V, \F)$, we need to compute the index invariant of $C^*(V,\F)$  with respect to the constructed extension of the form \eqref{SES}. If considered $C^*$-algebras are isomorphic to the reduced crossed products of the form $C_0(M) \rtimes H$, where $M$ is some locally compact space and $H$ is some Lie group acting on $M$, we can use the Thom-Connes isomorphism (\cite{Co81}) to compute the connecting maps $\delta_0,\delta_1$.
	
	By Corollary \ref{Analytical description}, we just study $K$-theory for the leaf space of MD(5,3C)-foliations belong to the non-trivial $\mathscr{F}_2$-type and characterize their Connes' $C^*$-algebras by $KK$-functor. We choose the foliation $\left( V, \F_{8\left(1,\frac{\pi}{2}\right)}\right)$ deputising for the $\mathscr{F}_2$-type. By Lemma \ref{Connes-1}, we have $C^* \left( \mathscr{F}_2 \right) \cong C^* \left( V, \F_{8\left(1,\frac{\pi}{2}\right)}\right)$.

	Before we study the $K$-theory for the leaf spaces of MD(5,3C)-foliations belong to the $\mathscr{F}_2$-type, we need to choose some suitable saturated open submanifolds in foliated manifold and construct the extensions of the form \eqref{SES}. 

	Let  $U, W$ be the following submanifolds of $V$:
		\[
			\begin{array}{l}
				U := \{ (x;y;z;t;s) \in V: s \neq 0 \} \cong \R^2 \times \R^2 \times \left( \R \backslash \{0\} \right), \\
				W := V \backslash U = \{ (x;y;z;t;s) \in V: s = 0 \} \cong \R^2 \times \left( \R^2 \backslash \{0\} \right).
			\end{array}
		\]
It is easy to see that the action $\rho$ in formula \eqref{rho-action} preserves the subsets $U, W$.

Let $\iota$ and $\mu$ be the inclusions and the restrictions as follows:
	\begin{equation}\label{iota-mu}
		\begin{array}{l l}
	\iota: C_0 (U) \to C_0 (V), & \mu: C_0 (V) \to C_0(W)
\end{array}
	\end{equation}
where each function of $C_0(U)$ is extented to the one of $C_0(V)$ by taking zero-value outside $V$. It is obvious that $\iota$ and $\mu$ are $\rho$-equivariant. Moreover, the following sequence is equivariantly exact:
\begin{equation}\label{ESES}
        \xymatrix{0 \ar[r] & C_0(U)\ar[r]^{\iota} & C_0(V) \ar[r]^{\mu} & C_0(W) \ar[r] & 0},
\end{equation}
Now, we denote by $(U, \F)$ and $(W, \F)$ the restrictions of $\left( V, \F_{8\left(1,\frac{\pi}{2}\right)}\right)$ to $U$ and $W$, respectively. The following theorem is the first of the main results of the section.

\begin{thm}
$C^*(\mathscr{F}_2)$ admits the following canonical extension:
        \begin{equation}\label{Gamma}
				\xymatrix
					{
						0 \ar[r] & J \ar[r]^{\hspace{-15pt} \hat{\iota}} & C^*(\mathscr{F}_2) \ar[r]^{\hspace{12pt} \hat{\mu}} & B \ar[r] & 0
					}\tag{$\gamma$}
			\end{equation}
where:
\[
			\begin{array}{l}
			J  \cong C_0(U) \rtimes_\rho \R^2  \cong C_0 \left( \R^2 \times \R^2 \times \R^* \right) \rtimes_\rho \R^2  \cong C_0 \left( \R^3 \sqcup \R^3 \right) \otimes \K,\\			
			B  \cong C_0(W) \rtimes_\rho \R^2  \cong C_0 \left( \R^2 \times \left(\R^2 \backslash \{0\}\right) \right) \rtimes_\rho \R^2 \cong C_0(\R \times \R_+) \otimes \K,\\			
			C^*(\mathscr{F}_2)  \cong C_0(V) \rtimes_\rho \R^2.
			\end{array}
			\]
and the homomorphismes $\hat{\iota}$ and $\hat{\mu}$ are defined by:
       \[
        \begin{array}{l l l}
        	 \left({\hat{\iota}f} \right) (r, s) = \iota f (r, s), & \left({\hat{\mu}f} \right) (r, s) = \mu f (r, s); & (r, s) \in \R^2.
        \end{array}
        \]
\end{thm}    

\begin{proof}
	We note that the graph of $\left(U, \F\right)$ is given as $U \times \R^2$. So, by using Proposition \ref{Connes-2}, one has:
        $$J := C^*\left(U, \F\right) \cong C_0 \left(U\right) \rtimes_{\rho} \R^2.$$  
Similarly, we have $B := C^*(W,\F) \cong C_0 (W) \rtimes_{\rho} \R^2$. Note that $U$ and $W$ are saturated submanifolds in $V$. So we obtain extension \eqref{Gamma} by using Proposition \ref{Connes-5}. 

	On the other hand, it is easily seen that the foliation $(U,\F)$ can be derived from the submersion $p: U \cong \R^2 \times \R^2 \times \left( \R \backslash \{0\} \right) \to \R^3 \times \{-1;1\} \cong \R^3 \sqcup \R^3$ defined by:
		\[
				p(x;y;re^{i\theta};s) := \left( x-r \sin \theta; re^{i\theta}; {\rm sgn} (s) \right).
		\]
Hence, in view of Proposition \ref{Connes-3}, we get $J \cong C_0 \left( \R^3 \sqcup \R^3 \right) \otimes \K$. Similarly, the foliation $(W,\F)$ can be derived from the submersion $q: W \cong \R^2 \times \left( \R^2 \backslash \{0\} \right) \to \R \times \R_+$ defined by:
		\[
				q(x;y;re^{i\theta}) := (x-r \sin \theta; r).
		\]
Hence, we also have $B := C^* \left(W, \F \right) \cong C_0 \left(\R \times \R_+\right) \otimes \K$. The proof is complete.
\end{proof}

\subsection{Computing the invariant of $C^*(\mathscr{\F}_2)$}

	We now study the $K$-theory for the leaf spaces of MD(5,3C)-foliations belong to the $\mathscr{F}_2$-type and characterize $C^*(\mathscr{F}_2)$ by $K$-functors. Namely, we will compute the invariant of $C^*(\mathscr{F}_2)$ in $KK$-group of Kasparov (\cite{Kas}).

	First, we recall that extension \eqref{Gamma} defines $C^*(\mathscr{F}_2)$ as an element in $KK$-group $\Ext(B, J)$. This element is called the invariant of $C^*(\mathscr{F}_2)$. Namely, we have the following definition.

\begin{definition}
	The element $(\gamma)$ corresponding to the extension \eqref{Gamma} in $\Ext(B, J)$ is called \emph{the index invariant of $C^*(\mathscr{F}_2)$} and denoted by ${\rm Index} \, C^*(\mathscr{F}_2)$.
\end{definition}

	As the analyses in the introduction, ${\rm Index} \, C^*(\mathscr{F}_2)$ determines the so-called ``stable type'' of $C^*(\mathscr{F}_2)$ in $\Ext (B, J)$. To introduce the last main result, we need some lemmas as follows.

\begin{lemma}\label{Lemma1}
	Set $I := C_0 \left(\R^2 \times \left( \R \backslash \{0\} \right) \right)$ and $A := C_0 \left(\R^2 \backslash \{0\} \right)$. Then, the following diagram is commutative:
		\[
\xymatrix{\dots \ar[r] & K_j(I) \ar[r] \ar[d]^{\beta_2} & K_j\bigl(C_0(\R^3 \backslash \{0\}) \bigr) \ar[r] \ar[d]^{\beta_2} & K_j(A) \ar[r] \ar[d]^{\beta_2} & K_{j+1}(I) \ar[r] \ar[d]^{\beta_2} & \dots \\
            \dots \ar[r] & K_j \bigl(C_0(U)\bigr) \ar[r] & K_j\bigl(C_0(V)\bigr) \ar[r] & K_j \bigl(C_0(W)\bigr)\ar[r] & K_{j+1} \bigl(C_0(U)\bigr)\ar[r] & \dots}
            \]
where $\beta_2$ is the Bott isomorphism (see \cite[Theorem 9.7]{Ta} or \cite[Corollary VI.3]{Co81}), and $j \in \mathbb{Z}/2\mathbb{Z}$.
\end{lemma}

\begin{proof}
Let
	$$\kappa: I = C_0 \left(\R^2 \times \left( \R \backslash \{0\} \right) \right) \to C_0 \left( \R^3 \backslash \{0\} \right),
     \qquad \upsilon: C_0 \left( \R^3 \backslash \{0\} \right) \to A = C_0 \left(\R^2 \backslash \{0\} \right)$$
be the inclusion and restriction defined similarly as $\iota$ and $\mu$ in \eqref{iota-mu}. One gets the exact sequence as follows:
\begin{equation}
\xymatrix{0\ar[r] & I \ar[r]^{\hspace{-.9cm} \kappa} & C_0 \left( \R^3 \backslash \{0\} \right) \ar[r]^{\hspace{.9cm} \upsilon } & A \ar[r] & 0}.
\end{equation}
Note that
\[
					\begin{array}{l}
						C_0(V) = C_0 \left( \R^2 \times \left( \R^3 \backslash \{0\} \right) \right) \cong C_0 \left( \R^2 \right) \otimes C_0 \left( \R^3 \backslash \{0\} \right), \\
						C_0(U) = C_0 \left( \R^2 \times \R^2 \times \left( \R \backslash \{0\} \right) \right) \cong C_0 \left( \R^2 \right) \otimes I, \\
						C_0(W) = C_0 \left( \R^2 \times \left( \R^2 \backslash \{0\} \right) \right) \cong C_0 \left( \R^2 \right) \otimes A.
					\end{array}
				\]            
So the extension \eqref{ESES} can be identified to the following one:
\begin{equation}
 \xymatrix{0\ar[r] & C_0 \left( \R^2 \right) \otimes I \ar[r]^{\hspace{-.9cm} id \otimes \kappa} & C_0 \left( \R^2 \right) \otimes C_0 \left( \R^3 \backslash \{0\} \right) \ar[r]^{\hspace{.8cm}id \otimes \upsilon} & C_0 \left( \R^2 \right) \otimes A \ar[r] & 0}.
\end{equation}
Now, using \cite[Theorem 9.7; Corollary 9.8]{Ta} we obtain the assertion of the lemma.
\end{proof}

\begin{lemma}[{{\bf \cite[Lemma 3.3.6]{Vu3}}}]\label{Lemma2}
\begin{enumerate-1}
   \item[\rm (a)] $K_0(I) =0$, and $K_1(I) \cong \mathbb{Z}^2$.
   \item[\rm (b)] $K_0(A) \cong \mathbb{Z}$, and $K_1(A) \cong \mathbb{Z}$.
   \item[\rm (c)] $K_0 \left(C_0 \left(\R^3 \backslash \{0\} \right) \right) =0$, and $K_1 \left(C_0 \left(\R^3 \backslash \{0\} \right) \right) \cong \mathbb{Z}^2$ .
\end{enumerate-1}
\end{lemma}

	Now, we present the last main result of this section about the index invariant of $C^*(\mathscr{F}_2)$.

\begin{thm}[{{\bf The invariant index of $C^*(\mathscr{F}_2)$}}]
${\rm Index} \, C^*(\mathscr{F}_2) = \{\gamma\}$, where $\gamma = \begin{pmatrix} 1 \\ 1 \end{pmatrix}$ in the $KK$-group ${\rm Ext} \, (B, J) \equiv \Hom_\mathbb{Z} \left( \mathbb{Z}, \mathbb{Z}^2 \right)$.
\end{thm}

\begin{proof}
The extension \eqref{Gamma} gives rise to a six-term exact sequence:
\begin{equation}\label{SES-1}
\xymatrix{K_0(J)\ar[r] & K_0 \left(C^*(\mathscr{F}_2)\right) \ar[r] & K_0(B)\ar[d]^{\delta_0}\\
        K_1(B) \ar[u]^{\delta_1} & K_1 \left(C^*(\mathscr{F}_2)\right) \ar[l] & K_1(J)\ar[l]}
\end{equation}
By \cite[Theorem 4.14]{Ro}, we have:
$$
\gamma = (\delta_0, \delta_1) \in {\rm Ext} (B, J) \cong \Hom_{\mathbb{Z}} \left(K_0(B),K_1(J)\right) \oplus \Hom_{\mathbb{Z}} \left(K_1(B), K_0(J) \right).
$$
Since the Thom-Connes isomorphism commutes with $K$-theoretical exact sequence (see \cite[Lemma 3.4.3]{To}), we have the following commutative diagram $(j\in \mathbb{Z}/2\mathbb{Z})$:
	\[
        \xymatrix{\dots \ar[r] & K_j(J) \ar[r] & K_j \left(C^*(\mathscr{F}_2)\right) \ar[r] & K_j(B) \ar[r] & K_{j+1}(J)\ar[r] & \dots\\
        \dots \ar[r] & K_j \bigl(C_0(U)\bigr)\ar[r] \ar[u]^{\varphi_j}& K_j \bigl(C_0(V)\bigr)\ar[r]\ar[u]^{\varphi_j} & K_j\bigl(C_0(W) \bigr)\ar[r]\ar[u]^{\varphi_j} & K_{j+1}\bigl(C_0(U)\bigr)\ar[r]\ar[u]^{\varphi_{j+1}} & \dots }
	\]
In view of Lemma \ref{Lemma1}, the following diagram is commutative $ (j \in \mathbb{Z}/2\mathbb{Z})$:
	\[
        \xymatrix{
                \dots \ar[r] & K_j \bigl(C_0(U)\bigr)\ar[r] & K_j \bigl(C_0(V)\bigr)\ar[r] & K_j\bigl(C_0(W) \bigr)\ar[r] & K_{j+1}\bigl(C_0(U)\bigr)\ar[r] & \dots\\
        \dots \ar[r] & K_{j}(I)\ar[r]\ar[u]^{\beta_2} & K_{j} \left(C_0 \left(\R^3 \backslash \{0\} \right) \right) \ar[r]\ar[u]^{\beta_2} & K_{j}(A)\ar[r]\ar[u]^{\beta_2} & K_{j + 1}(I)\ar[r]\ar[u]^{\beta_2}& \dots}
	\]
Thus we can compute the pair $(\delta_0, \delta_1) \in \Hom_{\mathbb{Z}} \bigl(K_0(A), K_1(I)\bigr) \oplus \Hom_{\mathbb{Z}}\bigl(K_1(A), K_0(I)\bigr)$ instead of the pair $(\delta_0, \delta_1) \in \Hom_{\mathbb{Z}} \left(K_0(B),K_1(J)\right) \oplus \Hom_{\mathbb{Z}} \left(K_1(B), K_0(J) \right)$. By Lemma \ref{Lemma2}, the $K$-theory exact sequence \eqref{SES-1} can be identified with:
\begin{equation}
\xymatrix{0\ar[r] & 0 \ar[r] & \mathbb{Z} \ar[d]^{\delta_0}\\
        \mathbb{Z}\ar[u]^{\delta_1=0} & \mathbb{Z}^2 \ar[l] & \mathbb{Z}^2\ar[l]}
\end{equation}
By \cite[Theorem 6]{Vu3}, we have $\delta_0= \begin{pmatrix} 1 \\ 1 \end{pmatrix}$. The proof is complete.
\end{proof}

\section*{Acknowledgements} The authors wish to thank the University of Economics and Law, VNU-HCMC, the Ho Chi Minh City University of Education and the Ho Chi Minh City University of Physical Education and Sports for financial supports.

\end{document}